\documentclass[11pt,a4paper]{article}
\usepackage{amsmath,amssymb,amsthm,amsfonts,latexsym,graphicx,subfigure}
\usepackage[all,import]{xy}
\usepackage[numbers,sort&compress]{natbib}
\usepackage{mathrsfs}
\usepackage{amsmath}
\usepackage{indentfirst}
\usepackage{fancyhdr,enumerate}
\usepackage{bbm,color}
\usepackage{tikz}%mathpazo
\usetikzlibrary{decorations.pathreplacing,calc}
\usetikzlibrary{shapes.geometric, arrows}
\usepackage{accents,cases}
\usepackage[T1]{fontenc} % 确保支持扩展拉丁字符
\usepackage{lmodern} % 可选，提供更好的字体支持
\usepackage{multirow}%\usepackage{eps2pdf}
\usepackage[lined, ruled, linesnumbered, longend]{algorithm2e}
\usepackage{tikz}
\usetikzlibrary{shapes.geometric, arrows}
\usepackage[labelsep=space]{caption}

\usepackage{array,float}

\usepackage{hyperref}
\newcommand{\PreserveBackslash}[1]{\let\temp=\\#1\let\\=\temp}
\newcolumntype{C}[1]{>{\PreserveBackslash\centering}p{#1}}
\newcolumntype{R}[1]{>{\PreserveBackslash\raggedleft}p{#1}}
\newcolumntype{L}[1]{>{\PreserveBackslash\raggedright}p{#1}}

\makeatletter
\def\wbar{\accentset{{\cc@style\underline{\mskip8mu}}}}

\makeatother

\theoremstyle{plain}
\newtheorem{theorem}{Theorem}

\newtheorem{lemma}{Lemma}%[section]
\newtheorem{remark}{Remark}[section]

\allowdisplaybreaks

\begin{document}

\title{A new proof of the Lemmens--Seidel conjecture}
\author{Chuanyuan Ge\textsuperscript{a,\ensuremath{\ast}} \and Shiping Liu\textsuperscript{b} }

\footnotetext[0]{ \textsuperscript{*}Corresponding author.}
\footnotetext[0]{\textsuperscript{a}School of Mathematics and Statistics, Fuzhou University, Fuzhou 350108, China.\\
Email addresses:
{\tt cyge@fzu.edu.cn} }
\footnotetext[0]{\textsuperscript{b}School of Mathematical Sciences, 
University of Science and Technology of China, Hefei 230026, China. \\
Email addresses:
{\tt spliu@ustc.edu.cn}
}
%\footnotetext[3]{Max Planck Institute for Mathematics in the Sciences, Inselstrasse 22, 04103 Leipzig, Germany. \\
%Email addresses:
%{\tt dzhang@mis.mpg.de} \; and \; {\tt 13699289001@163.com}
%}
\date{}\maketitle
\begin{abstract}
In this paper, we give a new proof of the Lemmens-Seidel conjecture on the maximum number of equiangular lines with a common angle $\arccos(1/5)$. This conjecture was previously resolved by Cao, Koolen, Lin, and Yu in 2022 through an analysis involving forbidden subgraphs for the smallest Seidel eigenvalue $-5$. Our new proof is based on
bounds on eigenvalue multiplicities of graphs with degree no larger than $14$. To control the maximum degree of the graph associated with equiangular lines, we employ a recent inequality of Balla derived by matrix projection techniques.
Our strategy also leads to a new proof for the classical result obtained by Lemmens and Seidel in 1973 for the case where the common angle is $\arccos(1/3)$.

%Let $N_{\alpha}(d)$ denote the maximum number of equiangular lines with common angle $\arccos (\alpha)$ in $\mathbb{R}^d$. In 1973,  Lemmens and Seidel conjectured that $N_{\frac{1}{5}}(d)=276$ for $23\leq d\leq 185$ and  $N_{\frac{1}{5}}(d)=\lfloor\frac{3d-3}{2}\rfloor$ for $d\geq 185$. This conjecture was completely solved by Cao, Koolen, Lin, and Yu in 2022. In this paper, we give a new proof of the Lemmens-Seidel conjecture. Furthermore, we use our method to give a new proof of the result on $N_{\frac{1}{3}}(d)$ deduced by Lemmens and Seidel, which states $N_{\frac{1}{3}}(d)=28$ for $7\leq d\leq 15$ and $N_{\frac{1}{5}}(d)=2d-2$ for $d\geq 15$.

\end{abstract}

\section{Introduction}
 A set of lines in $\mathbb{R}^d$ passing through the origin is called equiangular if every two lines have a common angle $\arccos (\alpha)$. The study of equiangular lines and their variants is closely related to various research topics, see \cite{Balla-Draxler-Keevash-Sudakov-18,Cao-Koolen-Lin-Yu,balla2021equiangular,balla2024equiangular-exponential-regime,Jiang-Tidor-Yao-Zhang-Zhao-21} and the references therein.

Let $N(d)$ be the maximum number of equiangular lines in $\mathbb{R}^d$. An important problem raised by van Lint and Seidel \cite{van-Seidel-elliptic-geometry} is to determine $N(d)$ for any $d\geq 1$. Currently,  the value of $N(d)$ is known only for $d\leq 43$ with $d\notin\{18,19,20,21,42\}$. In the following table, we summarize the latest results on the values of $N(d)$ for $d\leq 43$ \cite{Greaves2023MOC,Greaves-Syatriad-2024-JCTA}.
\[
\resizebox{0.95\textwidth}{!}{$
\begin{array}{c|ccccccccccccccccc}
d 
& 2 & 3 & 4 & 5 & 6 & 7\text{--}14 & 15 & 16 & 17 
& 18 & 19 & 20 & 21 & 22 & 23\text{--}41 & 42 & 43 \\
\hline
N(d) 
& 3 & 6 & 6 & 10 & 16 & 28 & 36 & 40 & 48 
& 57\text{--}59 & 72\text{--}74 & 90\text{--}94 & 126 & 176 & 276 & 276\text{--}288 & 344
\end{array}
$}
\]
Gerzon \cite{Lemmens-Seidel-73} showed that $N(d)\leq \frac{d(d+1)}{2}$, which was later proved to be tight for $d=2,3,7,$ and $23$. It is still open whether Gerzon's bound is tight for infinitely many $d$ or not. Gillespie \cite[Conjecture 8.4]{gillespie2018equiangular} conjectured that Gerzon's bound is attained only in dimensions $2,3,7,$ and $23$. In \cite{de-00}, de Caen constructed an equiangular $\frac{2(d+1)^2}{9}$-set in $\mathbb{R}^d$ for $d=3\cdot2^{t-1}-1, t\geq 1$. This implies that $N(d)\geq cd^2$ for any $d\geq 1$, where $c$ is an absolute constant. For more constructions, the readers are referred to \cite{Greaves-Koolen-Munemasa-16,Barg-Yu-14,Jedwab-Wiebe-15}.

For $\alpha\in (0,1)$, we denote by $N_{\alpha}(d)$ the maximum number of equiangular lines in $\mathbb{R}^d$ with a common angle $\arccos(\alpha)$. Lemmens and Seidel  \cite{Lemmens-Seidel-73} derived the following bound for $N(d)$:
\begin{equation}\label{ineq:N(d)}
    N(d)\leq \max\left\{N_{\frac{1}{3}}(d),N_{\frac{1}{5}}(d),\ldots,N_{\frac{1}{2m-1}}(d),\frac{4dm(m+1)}{(2m+1)^2-d}\right\},\ \ \text {for any integer}\ m>\frac{\sqrt{d}-1}{2}.
\end{equation}
This naturally motivates the investigation of $N_{\alpha}(d)$. In particular, Lemmens and Seidel \cite{Lemmens-Seidel-73} completely determined the values of $N_{\frac{1}{3}}(d)$ for any $d\geq 3$. They proved $N_{\frac{1}{3}}(d)=4,6,10,16$ for $d=3,4,5,6$, respectively, and the following theorem.
\begin{theorem}[{\cite[Theorem 4.5]{Lemmens-Seidel-73}}]\label{thm:main2}
    We have $N_{\frac{1}{3}}(d)=28$ for any $7\leq d\leq 15$ and $N_{\frac{1}{3}}(d)=2d-2$ for any $d\geq 15$.
\end{theorem}

Combining the inequality \eqref{ineq:N(d)}, Theorem \ref{thm:main2}, and related constructions (see \cite[Section 2]{Lemmens-Seidel-73}), Lemmens and Seidel proved that $N(d)=28$ for any $7\leq d\leq 13$, $N(15)=36$, $N(21)=126$, $N(22)=176$ and $N(23)=276$. Furthermore, Lemmens and Seidel \cite{Lemmens-Seidel-73} proposed a conjecture concerning $N_{\frac{1}{5}}(d)$, which was recently solved by Cao, Koolen, Lin, and Yu \cite{Cao-Koolen-Lin-Yu}.  
\begin{theorem}[\cite{Cao-Koolen-Lin-Yu}]\label{thm:main}
    We have $N_{\frac{1}{5}}(d)=276$ for $23\leq d\leq 185$ and  $N_{\frac{1}{5}}(d)=\left\lfloor\frac{3d-3}{2}\right\rfloor$ for $d\geq 185$. 
\end{theorem}
%Their main tool is forbidden subgraphs
%for the smallest Seidel eigenvalue $-5$, and their method relies on the conclusion that there are finitely many minimal graphs with spectral radius greater than $2$, see \cite[Theorem 5.3]{Cao-Koolen-Lin-Yu} and its proof.

For general $\alpha$, Bukh \cite{Bukh-16} proved that $N_{\alpha}(d)\leq C_{\alpha}d$ for any $d\geq 1$, where $C_{\alpha}$ is a constant depending only on $\alpha$. Balla, Dr\"axler, Keevash, and Sudakov \cite{Balla-Draxler-Keevash-Sudakov-18} showed that $\limsup_{d\to \infty}\frac{N_{\alpha}(d)}{d}$ is less than $1.93$ for $\alpha\in (0,1)\setminus\{\frac{1}{3}\}$ while $\lim_{d\to \infty}\frac{N_{\alpha}(d)}{d}=2$ for $\alpha=\frac{1}{3}$. Jiang, Tidor, Yao, Zhang, and
 Zhao \cite{Jiang-Tidor-Yao-Zhang-Zhao-21} completely determined $\lim_{d\to \infty}\frac{N_{\alpha}(d)}{d}$ for all $\alpha\in (0,1)$ in terms of the so-called spectral radius order, which is a quantity introduced by Jiang and Polyanskii \cite{Jiang-Polyanskii-20}. In particular, their result tells that $\lim_{d\to\infty}\frac{N_{1/(2k-1)}(d)}{d}=\frac{k}{k-1}$ for any integer $k\geq 2$. In addition to the study of asymptotics of $N_{\alpha}(d)$, research on improving the upper bounds of $N_{\alpha}(d)$ has also been actively pursued, see \cite{balla2021equiangular,yu2017new}.

The proof of Theorem \ref{thm:main} by Cao, Koolen, Lin, and Yu \cite{Cao-Koolen-Lin-Yu} is based on an analysis involving forbidden subgraphs for the smallest Seidel eigenvalue $-5$.
In this paper, we present a new proof of Theorem \ref{thm:main}. The key steps are as follows. First, we prove an estimate for the multiplicity of eigenvalues of graphs whose second largest eigenvalue equals $2$ and whose maximum degree does not exceed $14$ (see Lemma \ref{lemma:multi}).
This estimate is based on a lemma from the authors' previous work \cite[Lemma 3]{ge-liu-2024equiangular}, which itself draws on the notion of discrete nodal domains and the upper bounds for their number obtained by Lin, Lippner, Mangoubi, and Yau \cite{lin2010nodalmulti}.
Secondly, we construct a spherical $\{-\alpha, \alpha\}$-code by appropriately selecting vectors spanning the equiangular lines, such that the associated graph satisfies all requirements of the multiplicity estimate lemma. The main challenge here is controlling the upper bound of the maximum degree, which we overcome by applying \citep[Theorem 2.1]{balla2021equiangular}, a result of Balla derived using matrix projection techniques. Balla \cite{balla2021equiangular} obtained a general bound on the maximum degree for equiangular lines with any common angle via this theorem; in our case, we refine this bound specifically for the angle $\arccos(1/5)$. This refinement is necessary for the proof of Theorem \ref{thm:main}, as discussed in Remark \ref{re:Delta}.

Compared with the previous proof \cite{Cao-Koolen-Lin-Yu}, a benefit of our approach is that it does not rely on the fact that the number of minimal graphs with spectral radius greater than $2$ is finite.
It is known that for $k\geq 4$, there are infinitely many minimal graphs with spectral radius greater than $k-1$ \cite[Theorem 1]{Jiang-Polyanskii-20}. Therefore, our approach might be useful in studying $N_{1/(2k-1)}(d)$ for $k\geq 4$. In Section \ref{sec:remark}, we discuss the main obstruction to extending the present method to the case $N_{1/7}(d)$.

%This observation suggests that our proof could potentially be extended to investigate
%From this perspective, our proof might have potentials to be generalized to the study of $N_{1/(2k-1)}(d)$ for $k\geq 4$. 
%For investigating $N_{1/(2k-1)}(d)$, we need to consider minimal graphs with spectral radius greater than $k-1$, but such graphs are infinite when $k\geq 4$ (see \cite[Theorem 1]{Jiang-Polyanskii-20}).

For $d\leq 7$, the value of  $N_{\frac{1}{3}}(d)$ follows directly from the relative bound, see \cite[Theorem 4.5]{Lemmens-Seidel-73}. However, it is more sophisticated to determine $N_{\frac{1}{3}}(d)$ for $d\geq 8$. In this case, our approach also works. We use our method to give a new proof of Theorem \ref{thm:main2}. 
\paragraph{Correction to the published version.}
In the published version of this paper, we stated that 
Cao et al.~\cite{Cao-Koolen-Lin-Yu} had proved the Lemmens--Seidel conjecture in full. More precisely, Cao et al.\ proved the conjecture for base sizes $3$ and $4$ (for the precise definition of the base size 
of a set of equiangular lines, see Definition 2.1 in 
\cite{Yoshino}). Yoshino \cite{Yoshino} subsequently proved the base-size-$5$ case. Combining the results of Lemmens and Seidel \cite{Lemmens-Seidel-73},
Cao et al.\cite{Cao-Koolen-Lin-Yu} and Yoshino \cite{Yoshino}, yields the Lemmens--Seidel conjecture in full. We apologize for the inaccurate attribution in the published version.

\section{Preliminaries}
\subsection{Notations and definitions}
 Let $G=(V,E)$ be a graph. For any $u,v\in V$, we use $u\sim v$ to denote $\{u,v\}\in E$. For any $u\in V$, we denote by $\mathrm{deg}_G(u)$ the degree of $u$ in $G$. Let $\Delta_G$ be the maximum degree of $G$. For any subset $V_0\subset V$, let $N_G(V_0)$ be the neighborhood of $V_0$ in $G$, that is, $N_G(V_0):=\{u\in V:\exists\  v\in V_0 \text{ such that }u\sim v\}$. Define $\overline{N}_G(V_0):=N_G(V_0)\cup V_0$. We denote by $G[V_0]$ the induced subgraph of $V_0$ in $G$.
 
%We write $K_n$ for the $n$-vertex complete graph and $P_n$ the $n$-vertex path graph. We define $\mathcal{T}_{3}$ to be the graph depicted in Figure \ref{fig}.

 %For any $n\times n$ symmetric matrix $M$, we denote its $i$-th largest eigenvalue by $\lambda_i(M)$. Thus, the eigenvalues of $M$ can be listed as follows $$\lambda_1(M)\geq \lambda_2(M)\geq \cdots\geq \lambda_n(M).$$Let $\lambda$ be an eigenvalue of $M$. We use $m_{\lambda}(M)$ to denote the multiplicity of $\lambda$ in $M$. 
We list the eigenvalues of the adjacency matrix $A_G$ of $G$ as follows
\[\lambda_1(A_G)\geq \lambda_2(A_G)\geq \cdots\geq \lambda_n(A_G),\]
where $n$ is the size of the graph. 
 We denote the adjacency matrix of $G$ by $A_G$. We write $m_{\lambda}(A_G)$ for the multiplicity of an adjacency eigenvalue $\lambda$.

 Let $I_n$ be the $n\times n$ identity matrix, and $J_n$ be the $n\times n$ all-$1$ matrix. We use $\mathbf{1}_n$ to denote the $n$-dimensional all $1$ vector. When the dimension is obvious from context, we drop the subscript and simply write $I,J, \mathbf{1}$ instead of $I_n,J_n,\mathbf{1}_n$.

 A spherical $\{-\alpha,\alpha\}$-code $\mathscr{C}$ in $\mathbb{R}^d$ is a set of unit vectors in $\mathbb{R}^d$ where the inner product between any two distinct vectors is either $\alpha$ or $-\alpha$. For a collection of vectors $\mathscr{C} = \{v_1, v_2, \dots, v_n\}$, the Gram matrix $M_{\mathscr{C}} = \{m_{ij}\}_{1 \leq i,j \leq n}$ is defined by $m_{ij} = \langle v_i, v_j \rangle$.
We associate a graph $G = (V, E)$ with a spherical $\{-\alpha, \alpha\}$-code $\mathscr{C}$ by setting
\begin{equation}\label{eq:graphs_associated_code}
V = \{v_1, v_2, \dots, v_n\}, \quad
E = \{\{v_i, v_j\} : \langle v_i, v_j \rangle = -\alpha\}.
\end{equation}
\subsection{Connection between equiangular lines and graphs}
Let $\{l_1,l_2,\ldots,l_n\}$ be $n$ equiangular lines in $\mathbb{R}^d$ with a common angle $\arccos(\alpha)$. Pick a set of unit vectors $\mathscr{C}=\{v_1,v_2,\ldots,v_n\}$ such that $v_i$ spans $l_i$. By definition, $\mathscr{C}$ is an $n$-spherical $\{-\alpha,\alpha\}$-code in $\mathbb{R}^d$. Let $G=(V,E)$ be the graph associated with $\mathscr{C}$ defined by \eqref{eq:graphs_associated_code}. We have \begin{equation}\label{eq:M_C}
\frac{M_{\mathscr{C}}}{2\alpha}=\frac{1-\alpha}{2\alpha}I+\frac{1}{2}J-A_G.
\end{equation}
This implies that the matrix $\frac{1-\alpha}{2\alpha}I+\frac{1}{2}J-A_G$ is positive semidefinite with \[\mathrm{rank}\left(\frac{1-\alpha}{2\alpha}I+\frac{1}{2}J-A_G\right)\leq d.\] Conversely, if there exists an $n$-vertex graph $G$ with the matrix $\frac{1-\alpha}{2\alpha}I+\frac{1}{2}J-A_G$ being positive semidefinite and $\mathrm{rank}\left(\frac{1-\alpha}{2\alpha}I+\frac{1}{2}J-A_G\right)\leq d$, then there exist $n$ equiangular lines in $\mathbb{R}^d$ with a common angle $\arccos (\alpha)$. This is due to the fact that any positive semidefinite matrix $P$ has a decomposition $P=C^\top C$. In summary we have the following lemma.
\begin{lemma}\label{lemma:equivalent}
    The following statements are equivalent.
\begin{itemize}
    \item [(1)] There exist $n$ equiangular lines in $\mathbb{R}^d$ with a common angle $\arccos (\alpha)$.
    \item [(2)] There exists an $n$-spherical $\{-\alpha,\alpha\}$-code in $\mathbb{R}^d$.
    \item [(3)] There exists an $n$-vertex graph such that the matrix $\frac{1-\alpha}{2\alpha}I+\frac{1}{2}J-A_G$ is positive semidefinite with $\mathrm{rank}\left(\frac{1-\alpha}{2\alpha}I+\frac{1}{2}J-A_G\right)\leq d$. 
\end{itemize}  
Furthermore, the adjacency matrix $A_G$ of the graph $G$ associated with any $n$-spherical $\{-\alpha,\alpha\}$-code $\mathscr{C}$ in $\mathbb{R}^d$ must satisfy $\mathrm{rank}(\frac{1-\alpha}{2\alpha}I+\frac{1}{2}J-A_G)\leq d$ and $\frac{1-\alpha}{2\alpha}I+\frac{1}{2}J-A_G$ is positive semidefinite. 
\end{lemma}
\subsection{Spectral radius and eigenvalue multiplicity}

Let $G=(V,E)$ be a graph. We say that two subsets $V_1$ and $V_2$ of $V$ are \emph{edge disjoint}, if $V_1\cap V_2=\emptyset$ and the set $\{\{u,v\}:u\in V_1,v\in V_2,u\sim v\}$ is empty. 
\begin{lemma}[{\cite[Lemma 2.2]{balla2024equiangular-exponential-regime}}]\label{lemma:edgedisjoint}
Let $G=(V,E)$ be a connected graph. Let $V_1$ and $V_2$ be two edge disjoint subsets of $V$. Then, we have $\lambda_1(A_{G[V_1]})<\lambda_2(A_G)$ or $\lambda_1(A_{G[V_2]})<\lambda_2(A_G)$ or $\lambda_1(A_{G[V_1]})=\lambda_1(A_{G[V_2]})=\lambda_2(A_G)$.
\end{lemma}
The \emph{cyclomatic number} $\ell_G$ of a connected graph $G=(V,E)$ is defined as $\ell_G:=|E|-|V|+1$. The next lemma provides an estimate for the multiplicity of eigenvalues in relation to the cyclomatic number. Its proof relies on an upper bound for the number of strong nodal domains established by Lin, Lippner, Mangoubi, and Yau \cite{lin2010nodalmulti}, along with a construction of specific eigenfunctions on trees. For details, see 
 \cite[Lemma 3]{ge-liu-2024equiangular}.
\begin{lemma}[{\cite[Lemma 3]{ge-liu-2024equiangular}}]\label{lemma:multi tree}
    For any connected graph $G$, we have $m_{\lambda_k(A_G)}(A_G)\leq (k-1)\Delta_G+\ell_G$.
\end{lemma}

\section{Proofs}\label{sec:radius-graph}
In this section, we present our new proof of Theorem \ref{thm:main2} and Theorem \ref{thm:main}.
\subsection{Multiplicity estimation for graphs with bounded degree}\label{subsec:prepare}
In this part, we provide an estimate for the multiplicity of eigenvalues for graphs that appear during the proof of Theorem \ref{thm:main}.

\begin{lemma}\label{lemma:multi}
    Let $G=(V,E)$ be a connected graph with $\lambda_2(A_G)=2$ and $\Delta_G\leq 14$. Then, we have $m_2(A_G)\leq 80$.
\end{lemma}
\begin{proof}
We first introduce two notations. Let $\mathcal{T}(1,1,1,2)$ denote the tree obtained from the star graph $K_{1,4}$ by attaching a new leaf to one of its leaves, and let $\mathcal{T}(2,2,2)$ denote the tree obtained from the star graph $K_{1,3}$ by attaching a new leaf to each of its three leaves. For clarity, these two trees are illustrated in Figure \ref{fig}.

The lemma is straightforward when $|V|\leq 80$.
From here on, we will assume $|V|\geq 81$. Our proof proceeds by examining two separate cases. 

Case 1: $\Delta_G\geq 4$.

In this case, the graph $\mathcal{T}(1,1,1,2)$ is a subgraph of $G$. It is straightforward to check that $\lambda_1\left(A_{\mathcal{T}(1,1,1,2)}\right)>2$. Let $V_0$ be the vertex set of $\mathcal{T}(1,1,1,2)$. Recall that $\overline{N}_G(V_0)=N_G(V_0)\cup V_0$.
Since $\Delta_G\leq 14$, we obtain $|\overline{N}_G(V_0)|\leq 80$. By Lemma \ref{lemma:edgedisjoint}, we have $\lambda_1\left({A_{G[V\setminus \overline{N}(V_0)]}}\right)<2$. Then, the Cauchy Interlacing Theorem implies that $m_2(A_{G})\leq 80$.
\begin{figure}[h!]
\centering
\begin{minipage}{0.45\textwidth}
\centering
\begin{tikzpicture}[scale=1]
\tikzstyle{vertex}=[circle,fill=black,inner sep=2pt]
% Nodes
\node[vertex] (c) at (0,0) {};
\node[vertex] (u) at (0,1.5) {};
\node[vertex] (d) at (0,-1.5) {};
\node[vertex] (l) at (-1.5,0) {};
\node[vertex] (r1) at (1.5,0) {};
\node[vertex] (r2) at (3,0) {};
% Edges
\draw (c) -- (u);
\draw (c) -- (d);
\draw (c) -- (l);
\draw (c) -- (r1) -- (r2);
\end{tikzpicture}
\caption*{(a) The graph $\mathcal{T}(1,1,1,2)$}
\end{minipage}
\begin{minipage}{0.45\textwidth}
\centering
\begin{tikzpicture}[scale=1]
\tikzstyle{vertex}=[circle,fill=black,inner sep=2pt]
% Nodes
\node[vertex] (c) at (0,0) {};
\node[vertex] (u1) at (1.5,1.5) {};
\node[vertex] (m1) at (1.5,0) {};
\node[vertex] (d1) at (1.5,-1.5) {};
\node[vertex] (u2) at (3,1.5) {};
\node[vertex] (m2) at (3,0) {};
\node[vertex] (d2) at (3,-1.5) {};
% Edges
\draw (c) -- (u1) -- (u2);
\draw (c) -- (m1) -- (m2);
\draw (c) -- (d1) -- (d2);
\end{tikzpicture}
\caption*{(b) The graph $\mathcal{T}(2,2,2)$}
\end{minipage}
\caption{Illustration of two graphs}
	\label{fig}
\end{figure}

\iffalse
\begin{figure}[!htp]
	\centering
	\tikzset{vertex/.style={circle, draw, fill=black!20, inner sep=0pt, minimum width=3pt}}	

	\begin{tikzpicture}[scale=1.0]
  \draw    (2,0) -- (3,0) node[midway, above, black]{$ $} -- (4,0) node[midway, above, black]{$ $};
    
  \draw    (2,0) -- (3,1) node[midway, above, black]{$ $} -- (4,1) node[midway, above, black]{$ $};

  \draw    (2,0) -- (3,-1) node[midway, above, black]{$ $} -- (4,-1) node[midway, above, black]{$ $};

\node at (2,0) [vertex, label={[label distance=1mm]270: \small $ $}, fill=black] {};

\node at (3,1) [vertex, label={[label distance=1mm]270: \small $ $}, fill=black] {};

\node at (4,1) [vertex, label={[label distance=1mm]270: \small $ $}, fill=black] {};

  \node at (3,0) [vertex, label={[label distance=1mm]270: \small $ $}, fill=black] {};

\node at (4,0) [vertex, label={[label distance=1mm]270: \small $ $}, fill=black] {};

   \node at (4,-1) [vertex, label={[label distance=1mm]270: \small $ $}, fill=black] {};

    \node at (3,-1) [vertex, label={[label distance=1mm]270: \small $ $}, fill=black] {};

    \node at (3,-2) [vertex, label={[label distance=1mm]270: \small (b) The graph $\mathcal{T}_3$}, fill=black] {};
	\end{tikzpicture}
	\caption{The graph $\mathcal{T}_3$}
	\label{fig}
\end{figure}
\fi

Case 2: $\Delta_G\leq 3$. 

If the cyclomatic number $\ell_G$ of $G$ is at most $1$, then we have $m_{2}(A_G)\leq 4$ by Lemma \ref{lemma:multi tree}. In the following, we assume $\ell_G\geq 2$. 

If the girth $\mathfrak{g}_G$ of $G$ is at most $5$, then there is a cycle $C_1$ in $G$ of length no larger than $5$. Let $u$ be a vertex in $V\setminus C_1$ adjacent to a vertex $v$ of $C_1$. Denote by $C_2:=C_1\cup \{u\}$. By definition, we have $\lambda_1\left(A_{G[C_2]}\right)>2$. Following Lemma \ref{lemma:edgedisjoint}, we obtain $\lambda_1\left(A_{G[V
\setminus\overline{N}( C_2)]}\right)<2$. Since $\Delta_G\leq 3$, we have $|\overline{N}_G(C_2)|\leq 12$. By the Cauchy Interlacing Theorem, we derive $m_{2}(A_G)\leq 12$. 

Finally, we consider the case $\mathfrak{g}_G> 5$. Since $\ell_G\geq 2$, the graph $\mathcal{T}(2,2,2)$ is a subgraph of $G$. Denote by $V_1$ the vertex set of $\mathcal{T}(2,2,2)$. Pick a vertex $u_1$ in $V\setminus V_1$ such that there is a vertex $v_1$ in $V_1$ adjacent to $u_1$. Define $V_1'=V_1\cup \{u_1\}$. By definition, we have $\lambda_1\left(A_{G[V_1']}\right)>2$. Hence, we have $\lambda_1\left(A_{G[V
\setminus \overline{N}(V_1')]}\right)<2$ by Lemma \ref{lemma:edgedisjoint}. Since $\Delta_G\leq 3$, we deduce that $|\overline{N}_G(V'_1)|\leq 18$. By the Cauchy Interlacing Theorem, we have $m_{2}(A_G)\leq 18$. 

This completes the proof of this lemma.
\end{proof}

In Lemma \ref{lemma:multi}, we assume that the maximum degree of the graph does not exceed $14$. We then prove that, for an $n$-spherical $\{-\frac{1}{5},\frac{1}{5}\}$-code satisfying specific requirements on $M_{\mathscr{C}}$ and $n$, the associated graph necessarily fulfills this degree bound.
\begin{lemma}\label{lemma:maximum-degree}
    Let $\mathscr{C}$ be an $n$-spherical $\{-\frac{1}{5},\frac{1}{5}\}$-code with $n\geq 276$. Assume that $\lambda_1(M_{\mathscr{C}})>12$ and there exists an eigenfunction of $M_{\mathscr{C}}$ corresponding to $\lambda_1(M_{\mathscr{C}})$ without negative coordinates. Then, the graph $G$ associated with $\mathscr{C}$ satisfies $$\Delta_G\leq 14.$$
\end{lemma}
\begin{remark}\label{re:Delta}
    For general $\alpha\in (0,1)$, Balla \cite[Lemma 3.8]{balla2021equiangular} proved the following results: For any $n$-spherical $\{-\alpha,\alpha\}$-code $\mathscr{C}$ with $n\geq \frac{(1-\alpha^2)(1-2\alpha^2)}{2\alpha^4}$, if $\lambda_1(M_{\mathscr{C}})>\frac{1-\alpha^2}{2\alpha^2}$ and there exists an eigenfunction of $M_{\mathscr{C}}$ corresponding to $\lambda_1(M_{\mathscr{C}})$ without negative coordinates, then the maximum degree $\Delta_G$ of the graph $G$ associated with $\mathscr{C}$ satisfies $\Delta_G<\frac{1+3\alpha^2-4\alpha^3}{8\alpha^3}$. In the cases $\alpha=\frac{1}{3}$ and $\alpha=\frac{1}{5}$, Balla's result implies $\Delta_G<4$ and $\Delta_G<17$. Our Lemma \ref{lemma:maximum-degree} strengthens Balla's bound in the case $\alpha=\frac{1}{5}$. To prove Theorem \ref{thm:main}, this strengthening is necessary. However, to prove Theorem \ref{thm:main2}, Balla's bound is enough.
\end{remark}
The primary tool for proving Lemma \ref{lemma:maximum-degree} is the following theorem of Balla \cite{balla2021equiangular}, derived from the method of matrix projections.
\begin{theorem}{{\cite[Theorem 2.1]{balla2021equiangular}}}\label{thm:projection}
    Let $\alpha\in (0,1)$ and $\mathscr{C}$ be an $n$-spherical $\{-\alpha,\alpha\}$-code in $\mathbb{R}^d$. If $x$ is a unit eigenfunction of $M_{\mathscr{C}}$ corresponding to an eigenvalue $\lambda$, then we have $$\lambda\left(\frac{\lambda^2}{\alpha^2n+1-\alpha^2}-\frac{1-\alpha^2}{2\alpha^2}\right)x(u)^2\geq \lambda-\frac{1-\alpha^2}{2\alpha^2},$$
    for any $u\in \mathscr{C}.$
\end{theorem}
\begin{remark}\label{rem:positive}
    By the above theorem, we have $\frac{\lambda^2}{\alpha^2n+1-\alpha^2}-\frac{1-\alpha^2}{2\alpha^2}>0$ whenever $\lambda>\frac{1-\alpha^2}{2\alpha^2}$.
\end{remark}

\begin{proof}[Proof of Lemma \ref{lemma:maximum-degree}]
   Denote $\lambda=\lambda_1(M_{\mathscr{C}})$. Let $x$ be a unit eigenfunction of $M_{\mathscr{C}}$ corresponding to $\lambda$ without negative coordinates.  By Theorem \ref{thm:projection}, we have $$x(u)^2\geq \frac{1-\frac{12}{\lambda}}{\frac{25\lambda^2}{n+24}-12},$$ for any $u\in \mathscr{C}$. Since $\lambda>12,$ we have $\frac{25\lambda^2}{n+24}-12>0$ by Theorem \ref{thm:projection}, see Remark \ref{rem:positive}. 
   Since all coordinates of $x$ are non-negative, we deduce \begin{equation}\label{eq:low}
       x(u)\geq \sqrt{\frac{1-\frac{12}{\lambda}}{\frac{25\lambda^2}{n+24}-12}}=:Q(\lambda,n),
   \end{equation}
   for any $u\in \mathscr{C}$. Using \eqref{eq:M_C} and \eqref{eq:low}, we compute
\begin{equation}\label{eq:upper}
\begin{aligned}
       (5\lambda-4)Q(\lambda,n)\leq & (5\lambda-4)x(u)=((5M_{\mathscr{C}}-4I)x)(u)\\
       =&((J-2A_G)x)(u)=\langle \mathbf{1},x\rangle-2\sum_{v\in N_G(u)}x(v)\leq \sqrt{n}-2d_G(u)Q(\lambda,n).
\end{aligned}
\end{equation}
Rearranging (\ref{eq:upper}) yields that $$d_G(u)\leq \frac{\sqrt{n}}{2Q(\lambda,n)}-\frac{5\lambda}{2}+2,$$
for any $u\in V$. It is enough to show for $\lambda>12$ that \begin{equation}\label{eq:12}
   \frac{\sqrt{n}}{2Q(\lambda,n)}-\frac{5\lambda}{2}=\frac{1}{2}\sqrt{\frac{f(n)}{1-\frac{12}{\lambda}}}-\frac{5\lambda}{2}\leq 12,
\end{equation}
where $f$ is a function defined via $f(t):=\frac{25t\lambda^2}{t+24}-12t$. We prove the inequality (\ref{eq:12}) by considering  two separate cases. Observe that $f$ is monotonically increasing when $0\leq  t\leq 5\sqrt{2}\lambda-24$ and monotonically decreasing when $t\geq 5\sqrt{2}\lambda-24$.

Case 1. $\lambda< 40$. We have $5\sqrt{2}\lambda-24<276\leq n.$  Hence, we have  $$f(n)\leq f(276)=23\lambda^2-3312.$$Thus, it suffices to show 
\begin{equation}\label{eq:12leq40}
    \frac{1}{2}\sqrt{\frac{23\lambda^2-3312}{1-\frac{12}{\lambda}}}-\frac{5\lambda}{2}\leq 12.
\end{equation}
Because $\lambda> 12$, we have the inequality (\ref{eq:12leq40}) is equivalent to
\begin{equation}\label{eq:12l40}
  g(\lambda):=\lambda^3-30\lambda^2+504\lambda-3456\geq 0.
\end{equation}
 Observe that $g(\cdot)$ is monotonically increasing. Then, we obtain $ \lambda^3-30\lambda^2+504\lambda-3456\geq g(12)=0$. That is, \eqref{eq:12l40} holds and hence we derive \eqref{eq:12}.

Case 2. $\lambda\geq 40$.  This implies that $$f(n)\leq f(5\sqrt{2}\lambda-24)= 25\lambda^2-120\sqrt{2}\lambda+288.$$ 
It is enough to prove
\begin{equation*}
    \frac{1}{2}\sqrt{\frac{25\lambda^2-120\sqrt{2}\lambda+288}{1-\frac{12}{\lambda}}}-\frac{5\lambda}{2}\leq 12.
\end{equation*}
Since $\lambda\geq 40$, the above inequality is equivalent to 
\begin{equation}\label{eq:12g40}
    (120\sqrt{2}-60)\lambda^2-2592\lambda-6912\geq 0.
\end{equation}
Since \eqref{eq:12g40} holds for  $\lambda\geq 40$, we obtain \eqref{eq:12}.

Combining Case 1 and Case 2, we conclude the proof of this lemma.
\end{proof}

To apply Lemma \ref{lemma:maximum-degree}, we require that the largest eigenvalue of the Gram matrix to exceed $12$ and to be associated with an eigenfunction without negative coordinates. The next two lemmas confirm that these requirements are met.
\begin{lemma}\label{lemma:nonnegative}
    Let $\{l_1,l_2,\ldots,l_n\}$ be $n$ equiangular lines in $\mathbb{R}^d$ with a common angle $\arccos(\alpha)$. Then, there exists an $n$-spherical $\{-\alpha,\alpha\}$-code $\mathscr{C}$ in $\mathbb{R}^d$, such that the Gram matrix $M_{\mathscr{C}}$ has an eigenfunction  corresponding to the largest eigenvalue, whose coordinates are all non-negative. 
\end{lemma}
\begin{proof}
    Pick a set of unit vectors $\mathscr{C}_0=\{u_1,u_2,\ldots,u_n\}$ such that $u_i$ spans $l_i$ for $1\leq i\leq n$. Let $x$ be an eigenfunction of $M_{\mathscr{C}_0}$ corresponding to the largest eigenvalue. Denote by $$I:=\{i\in \{1,2,\ldots,n\}:x(u_i)<0\}.$$ We define a new set of unit vectors $\mathscr{C}=\{v_1,v_2,\ldots,v_n\}$ as follows $$v_i=\begin{cases}
u_i & \text{if } i\notin I,\\
 -u_i & \text{if } i \in I,
 \end{cases}$$
 and a new function $y$ such that $y(v_i):=|x(u_i)|$ for $1\leq i\leq n$. It is straightforward to check that $\mathscr{C}$ is an $\{-\alpha,\alpha\}$-code and $y$ is an eigenfunction of $M_{\mathscr{C}}$ corresponding to the largest eigenvalue. All coordinates of $y$ are non-negative. 
\end{proof}
\begin{lemma}\label{lemma:uniform-upper}
    For any $\alpha\in(0,1)$, let $\mathscr{C}$ be an $n$-spherical $\{-\alpha,\alpha\}$-code. If the largest eigenvalue satisfies $\lambda_1(M_{\mathscr{C}})\leq \frac{1-\alpha^2}{2\alpha^2}$, then $$n\leq \frac{(1-\alpha^2)(1-2\alpha^2)}{2\alpha^4}.$$ 
\end{lemma}
\begin{proof}
 This lemma follows by a direct calculation $$n(\alpha^2n+1-\alpha^2)=\mathrm{tr}(M^2_{\mathscr{C}})=\sum_{i=1}^n\lambda^2_i(M_{\mathscr{C}})\leq \lambda_1(M_{\mathscr{C}})\mathrm{tr}(M_{\mathscr{C}})\leq \frac{(1-\alpha^2)n}{2\alpha^2}.$$
 In the above, we use the fact that every $\lambda_i\geq 0$ since $M_{\mathscr{C}}$ is positive-semidefinite.
\end{proof}
\subsection{Proof of Theorem \ref{thm:main}}
 We first recall a lower bound for $N_{\frac{1}{2k-1}}(d)$.
\begin{lemma}\label{lemma:lowerbound}
   For any $d\geq 1$ and $k\geq 2$, we have $N_{\frac{1}{2k-1}}(d)\geq \left\lfloor\frac{k(d-1)}{k-1}\right\rfloor.$
\end{lemma}
\begin{proof}
    By Lemma \ref{lemma:equivalent}, it is sufficient to prove that there exists a $\left\lfloor\frac{k(d-1)}{k-1}\right\rfloor$-vertex graph $G$, for which the matrix $(k-1)I+\frac{1}{2}J-A_G$ is positive semidefinite with $\mathrm{rank}\left((k-1)I+\frac{1}{2}J-A_G\right)\leq d$. Define a $\left\lfloor\frac{k(d-1)}{k-1}\right\rfloor$-vertex graph $G$ as the disjoint union of $\left\lfloor\frac{d-1}{k-1}\right\rfloor$ copies of complete graphs $K_{k}$ and $\left\lfloor\frac{k(d-1)}{k-1}\right\rfloor-k\left\lfloor\frac{d-1}{k-1}\right\rfloor$ isolated vertices. Since $\lambda_1(K_{k})=k-1$, we see $(k-1)I-A_G$  is positive semidefinite and $\mathrm{dim\,ker}((k-1)I-A_G)=\left\lfloor\frac{d-1}{k-1}\right\rfloor$. Since $J$ is positive semidefinite with rank $1$, we obtain $(k-1)I+\frac{1}{2}J-A_G$ is positive semidefinite and $$\mathrm{rank}\left((k-1)I+\frac{1}{2}J-A_G\right)\leq \mathrm{rank}((k-1)I-A_G)+1= \left\lfloor\frac{k(d-1)}{k-1}\right\rfloor-\left\lfloor\frac{d-1}{k-1}\right\rfloor+1=d,$$
    where we used the rank-nullity theorem.
\end{proof}
\begin{lemma}\label{lemma:secondeigenvalue}
    Assume $k\geq2$ is an integer. Let $\mathscr{C}$ be an $n$-spherical $\{-\frac{1}{2k-1},\frac{1}{2k-1}\}$-code in $\mathbb{R}^d$ and $G$ be the associated graph. If $n> \left\lfloor \frac{k(d-1)}{k-1} \right\rfloor$ and $d>3k-2$, then we have $$\lambda_1(A_G)>\lambda_2(A_G)=k-1,\,\,\text{ and }\,\,n\leq d+1+m_{\lambda_2(A_G)}(A_G).$$
\end{lemma}
\begin{proof}
We prove this lemma in $3$ steps.

    Step 1.  We prove that $k-1$ is an eigenvalue of $A_G$. Suppose $k-1$ is not an eigenvalue of $A_G$. Then we derive from  Lemma \ref{lemma:equivalent} that 
    $$d\geq \mathrm{rank}\left((k-1)I-A+\frac{1}{2}J\right)\geq \mathrm{rank}((k-1)I-A)-1\geq  n-1.$$
Since $d>3k-2$, we derive $n\leq d+1<\left\lfloor \frac{k(d-1)}{k-1}\right\rfloor$. Contradiction.

 Step 2. We prove that $k-1$ is not the largest eigenvalue of $A_G$. Suppose $k-1$ is the largest eigenvalue of $A_G$.
Denote by $\{G_i=(V_i,E_i)\}_{i=1}^t$ the connected components of $G$.
 We assume $\lambda_1(A_{G_i})=k-1$ for $1\leq i\leq l$ and  $\lambda_1(A_{G_i})<k-1$ for $l+1\leq i\leq t$. Then, we have \begin{equation}\label{eq:ngeq3l}
     n\geq \sum_{i=1}^l|V_i|\geq kl,
\end{equation}
 where we used the fact that any graph $G'=(V',E')$ with $\lambda_1(A_{G'})=k-1$ satisfies $|V'|\geq k$. Because both $(k-1)I-A_G$ and $\frac{1}{2}J$ are positive semidefinite, it follows that $$\mathrm{ker}\left((k-1)I-A+\frac{1}{2}J\right)=\mathrm{ker}((k-1)I-A)\cap\mathrm{ker}(J).$$
 By the Perron–Frobenius theorem, there exists a vector $x$ in $\mathrm{ker}((k-1)I-A)$ without negative coordinates. Because $x\notin \mathrm{ker}(J)$, we have  $$\mathrm{dim\,ker}\left((k-1)I-A+\frac{1}{2}J\right)\leq \mathrm{dim\,ker}((k-1) I-A)-1.$$
 By the rank-nullity
theorem, we deduce 
$$\mathrm{rank}((k-1) I-A)\leq\mathrm{rank}\left((k-1) I-A+\frac{1}{2}J\right)-1\leq d-1.$$
 Applying the inequality (\ref{eq:ngeq3l}), this leads to  $$n=\mathrm{rank}((k-1)I-A)+\mathrm{dim\,ker}((k-1) I-A)\leq d-1+\frac{n}{k}.$$
 That is, we have $n\leq \frac{k(d-1)}{k-1}$. Contradiction.

 Step 3. Now, we have shown that $k-1$ is an eigenvalue of $A_G$ and $\lambda_1(A_{G})>k-1$. Since $(k-1) I-A+\frac{1}{2}J$ is positive semidefinite and $\mathrm{rank}(J)=1$, we have $\lambda_2(A_{G})=k-1$.
Thus, we derive \begin{equation*}
\begin{aligned}
    n&=\mathrm{rank}((k-1)I-A_G)+\mathrm{dim\,ker}((k-1)I-A_G)\\
    &\leq\mathrm{rank}\left((k-1)I-A_G+\frac{J}{2}\right)+1+\mathrm{dim\,ker}((k-1)I-A_{G})\\
     &\leq d+1+m_{k-1}(A_G).
\end{aligned}
\end{equation*}

This concludes the proof of this lemma.
\end{proof}

Now, we prove Theorem \ref{thm:main}.
\begin{proof}[Proof of Theorem \ref{thm:main}]
  Conway \cite{Conway-69-characterisation} constructed an equiangular $276$-set in $\mathbb{R}^{23}$ with a common angle $\arccos(1/5)$. This tells particularly for $23\leq d\leq 185$ that
  $N_{\frac{1}{5}}(185)\geq N_{\frac{1}{5}}(d)\geq N_{\frac{1}{5}}(23)\geq 276.$
  Observe that $\left.\left\lfloor\frac{3d-3}{2}\right\rfloor\right\vert_{d=185}=276$. To prove Theorem \ref{thm:main}, it is enough to show $N_{\frac{1}{5}}(d)=\left\lfloor\frac{3d-3}{2}\right\rfloor$ for any $d\geq 185$. By Lemma \ref{lemma:lowerbound}, it remains to prove 
  \begin{equation}\label{eq:Aim}
      N_{\frac{1}{5}}(d)\leq  \left\lfloor \frac{3d-3}{2} \right\rfloor,\ \  \text{for any}\ d\geq 185.
  \end{equation}
  We prove this inequality by contradiction. Assume $N_{\frac{1}{5}}(d)> \left\lfloor \frac{3d-3}{2} \right\rfloor\geq 276$ for some $d\geq 185$.
    
Let $\{l_1,l_2,\ldots,l_n\}$ be $n$ equiangular lines in $\mathbb{R}^d$ with $n=N_{\frac{1}{5}}(d)$. By Lemma \ref{lemma:nonnegative}, there exists a spherical $\{-\frac{1}{5},\frac{1}{5}\}$-code $\mathscr{C}=\{v_1,v_2,\ldots,v_n\}$ in $\mathbb{R}^d$, such that the Gram matrix $M_{\mathscr{C}}$ has an eigenfunction corresponding to its largest eigenvalue without a negative coordinate. Let $G=(V,E)$ be the graph associated with $\mathscr{C}$. 
Applying Lemma \ref{lemma:uniform-upper} in the case $\alpha=\frac{1}{5}$ yields that $\lambda_1(M_{\mathscr{C}})>12$. Hence we obtain $\Delta_G\leq 14$ by Lemma \ref{lemma:maximum-degree}. Moreover, we derive $\lambda_1(A_G)>\lambda_2(A_G)=2$  by applying Lemma \ref{lemma:secondeigenvalue} in the case that $k=3$.

Denote by $\{G_i=(V_i,E_i)\}_{i=1}^t$ the connected components of $G$. There exists at least one connected component, say $G_1$, with $\lambda_1(A_{G_1})=\lambda_1(A_G)>2$. We claim that all the other connected components do not have this property. That is, we claim that $\lambda_1(G_i)<2$ for every $2\leq i\leq t$. 

 Let $f$ (resp., $g$) be an eigenfunction of $A_{G}$ corresponding to $\lambda_1(A_G)=\lambda_1(G_1)$ (resp., $\lambda_1(A_{G_i})$ for some $2\leq i\leq t$). By the Perron--Frobenius theorem, we can assume both $f$ and $g$ have no negative coordinates, and $f$ is positive on $V_1$ and $g$ is positive on $V_i$. Take a constant $c$ such that the function $h:=f-cg$ satisfies $h^T\mathbf{1}=0$. Notice that such a $c$ is non-zero. Since $2I-A_G+\frac{1}{2}J$ is positive semidefinite, we have $h^T(2I-A_G+\frac{1}{2}J)h\geq0$, which is equivalent to $$(2-\lambda_1(A_G))f^Tf\geq c^2(\lambda_1(A_{G_i})-2)g^Tg.$$
 This implies $\lambda_1(A_{G_i})< 2$ and proves the claim.

 The claim implies particularly that $m_{2}(A_G)=m_{2}(A_{G_1})$. Then, we deduce from  Lemma \ref{lemma:secondeigenvalue} that
\begin{equation*}
    n\leq d+1+m_2(A_{G})= d+1+m_2(A_{G_1})\leq d+1+80\leq \left\lfloor\frac{3d-3}{2}\right\rfloor,
\end{equation*}
where we have applied Lemma \ref{lemma:multi} to $G_1$. Contradiction. Therefore, we prove \eqref{eq:Aim}. 
\end{proof}
\subsection{Proof of Theorem \ref{thm:main2}}
In this subsection, we prove Theorem \ref{thm:main2} using the same approach employed in the proof of Theorem \ref{thm:main}. Our first step is to derive an estimate for the multiplicity of the eigenvalues.
\begin{lemma}\label{multi2}
    Let $G=(V,E)$ be a connected graph with $\lambda_2(A_G)=1$ and $\Delta_G\leq 3$. Then we have $m_1(A_G)\leq 8$.
\end{lemma}
\begin{proof}
    We can assume that $|V|\geq 9$. This implies that the path graph $P_3$ with $3$ vertices is a subgraph of $G$. By direct computation, we have $\lambda_1(A_{P_3})>1$. Denote by $V_3$ the vertex set of $P_3$. Because $\Delta_G\leq 3$, it follows that $|\overline{N}_G(V_3)|\leq 8$. By Lemma \ref{lemma:edgedisjoint}, we have $\lambda_1\left(A_{G[V\setminus\overline{N}_G(V_3)]}\right)<1$. By the Cauchy Interlacing Theorem, we deduce that $m_1(A_G)\leq |\overline{N}_G(V_3)|\leq 8.$
\end{proof}
Now, we prove Theorem $\ref{thm:main2}$.
\begin{proof}[Proof of Theorem \ref{thm:main2}]
In \cite{van-Seidel-elliptic-geometry,Seidel-67-Strongly-regular-graphs}, van Lint and Seidel 
constructed an equiangular $28$-set in $\mathbb{R}^{7}$ with a common angle $\arccos(1/3)$. By a similar argument as in the beginning of the proof of Theorem \ref{thm:main}, it is enough to show
$N_{\frac{1}{3}}(d)=2d-2$ for any $d\geq 15$. By Lemma \ref{lemma:lowerbound}, it remains to prove 
\begin{equation}\label{eq:Aim2}
    N_{\frac{1}{3}}(d)\leq 2d-2,\ \ \text{for any}\ d\geq 15. 
\end{equation}
We prove it by contradiction. Assume $N_{\frac{1}{3}}(d)> 2d-2\geq 28$ for some $d\geq 15$. Then, by Lemma \ref{lemma:equivalent} and Lemma \ref{lemma:nonnegative}, there exists an $n$-spherical $\{-\frac{1}{3},\frac{1}{3}\}$-code $\mathscr{C}$, such that the Gram matrix $M_{\mathscr{C}}$ has an eigenfunction corresponding to $\lambda_1(M_{\mathscr{C}})$ without negative coordinates. By Lemma \ref{lemma:uniform-upper} and Balla \cite[Lemma 3.8]{balla2021equiangular}, see  Remark $\ref{re:Delta}$, the graph $G$ associated with $\mathscr{C}$ satisfies $\Delta_G\leq 3$. Applying Lemma \ref{lemma:secondeigenvalue} in the case that $k=2$ yields that $\lambda_1(A_G)>\lambda_2(A_G)=1$.

Denote by $\{G_i=(V_i,E_i)\}_{i=1}^t$ the connected components of $G$. Assume $\lambda_1(A_{G_1})=\lambda_1(A_G)>1$. Similarly as in the proof of Theorem \ref{thm:main}, we have $\lambda_1(G_i)<1$ for any $2\leq i\leq t$. Then, we have $m_1(A_G)=m_1(A_{G_1})$. By Lemma \ref{lemma:secondeigenvalue}, we obtain
\begin{equation*}
    n\leq d+1+m_1(A_{G})= d+1+m_1(A_{G_1})\leq d+1+8\leq 2d-2,
\end{equation*}
where we have applied Lemma \ref{multi2} to $G_1$. Contradiction. This completes the proof of \eqref{eq:Aim2}. 
\end{proof}

\section{Concluding remarks}\label{sec:remark}
In this paper, we provide new proofs for the results that $ N_{1/5}(d)=\left\lfloor (3d-3)/2\right\rfloor
$ for $d\ge 185$, and $
N_{1/3}(d)=2d-2$ for $d\ge 15$. Since our approach does not rely on the fact that the number of minimal graphs with spectral radius greater than $2$ is finite, it is natural to ask whether our approach can be applied to  determine the least dimension $d_0$ such that
$N_{1/7}(d)=\left\lfloor (4d-4)/3\right\rfloor$
holds for all $d\ge d_0$.

For that purpose, we hope for a multiplicity estimate similar to Lemma \ref{lemma:multi} for connected graphs $G$ with $\lambda_2(A_G)=3$ and bounded maximal degree $\Delta_G$. Using a similar argument as in Case 1 of the proof for Lemma \ref{lemma:multi}, one can derive such an estimate in case $\Delta_G\geq 9$. 
For case $\Delta_G\leq 8$, Lemma \ref{lemma:multi tree} applies directly when the cyclomatic number $\ell_G$ is small to derive a multiplicity estimate. The difficulty is to deal with graphs with $\ell_G$ large and $\Delta_G\leq 8$. In order to extend our argument in Lemma \ref{lemma:multi}, we hope to find a finite family of graphs $\{G_1,\ldots, G_k\}$ with spectral radius greater than $3$, such that every graph $G$ with $\ell_G$ large and $\Delta_G\leq 8$ must contain at least one of $\{G_1,\ldots, G_k\}$ as a subgraph. However, we did not find such a family.  

Another possible way to overcome this difficulty is to strengthen Lemma \ref{lemma:multi tree}. 
In our previous work \cite{ge-liu-2024equiangular}, Lemma \ref{lemma:multi tree} was proved using theory of discrete nodal domains. The main idea is that, if the $k$-th largest eigenvalue has multiplicity $m$, then one can construct an eigenfunction corresponding to this eigenvalue whose number of nodal domains is at least $m-\ell_G$. Combining this lower bound with the upper bound $(k-1)\Delta_G$ of the number of nodal domains due to Lin, Lippner, Mangoubi, and Yau \cite{lin2010nodalmulti} yields the desired multiplicity estimate. This approach may admit further refinements. In particular, one possible refinement would be to construct an eigenfunction in the corresponding eigenspace with more than $m-\ell_G$ nodal domains. Such a refinement would lead to sharper multiplicity estimates.

\section*{Acknowledgements}
This work is supported by the National Key R and D Program of China 2020YFA0713100, the National Natural Science Foundation of China (No.12431004). We are very grateful to Mengyue Cao and Jack H. Koolen for helpful discussions.

\bibliographystyle{plain}
\bibliography{references}
\end{document}